\PassOptionsToPackage{dvipdfmx}{graphicx}

\makeatletter
\disable@package@load{program}{}
\makeatother

\documentclass[sn-mathphys-num]{sn-jnl}

\usepackage{ascmac}

\usepackage{amsmath}
\usepackage{amssymb}
\usepackage{amsthm}
\usepackage{enumerate}
\usepackage[T5,T1]{fontenc}
\usepackage{graphicx}
\usepackage[all]{xy}

\numberwithin{equation}{section}
\newcommand{\Gr}{\operatorname{Gr}}
\newcommand{\trianglerightneq}{\mathrel{\ooalign{\raisebox{-0.5ex}{\reflectbox{\rotatebox{90}{$\nshortmid$}}}\cr$\triangleright$\cr}\mkern-3mu}}
\newcommand{\Gal}{\operatorname{Gal}}

\theoremstyle{plain}
\newtheorem{theo}{Theorem}[section]
\newtheorem{lemm}[theo]{Lemma}
\newtheorem{prop}[theo]{Proposition}
\newtheorem{coro}[theo]{Corollary}

\theoremstyle{definition}
\newtheorem{defi}[theo]{Definition}
\newtheorem{rema}[theo]{Remark}
\newtheorem{exam}[theo]{Example}

\makeatletter
\renewcommand{\p@enumii}{}
\makeatother
\makeatletter %
\def\@acknow{}%
\long\def\EarlyAcknow#1 \par{%
\def\@acknow{\abstractfont\subabstracthead*{Acknowledgments}%
#1\par}}%

\def\printabstract{
    \ifx\@abstract\empty\else\@abstract\fi\par%
    \ifx\@acknow\empty\else\@acknow\fi\par%
    }
\makeatother

\begin{document}

\author*[1]{\fnm{Koto} \sur{Imai}\,\orcid{https://orcid.org/0009-0008-2109-2035}}\email{imai.koto.541@gmail.com}
\affil*[1]{\orgdiv{Graduate School of Mathematical Sciences}, \orgname{The University of Tokyo}, \\\orgaddress{\street{3--8--1 Komaba}, \city{Meguro-ku}, \postcode{153--8914}, \state{Tokyo}, \country{Japan}}\\}

\title{Ramification groups of some finite Galois extensions of maximal nilpotency class over local fields of positive characteristic}

\abstract{
  We examine the ramification groups of finite Galois extensions over complete discrete valuation fields of equal characteristic $p>0$. Brylinski (1983) calculated the ramification groups in the case where the Galois groups are abelian. We extend the results of Brylinski to some non-abelian cases where the Galois groups are of order $\leq p^{p+1}$ and of maximal nilpotency class. 
}
\pacs[Mathematical Subject Classification]{11S15, 12F10}
\EarlyAcknow{
The author wishes to thank Professor Takeshi Saito for suggesting this problem, giving helpful advice and ideas, and pointing out mistakes during the preparation of this paper.  }

\maketitle

\section{Introduction}

Let $K$ be a complete discrete valuation field of equal characteristic $p>0$. Assume that the residue field $k$ of $K$ is perfect. Let $L/K$ be a finite Galois extension of $K$. Then we can define a filtration of the Galois group, called ramification groups, as in \cite{Se}, IV, \S 3. 

This filtration is a convenient tool for studying the wild ramification. If $\Gal(L/K)$ is abelian, the Hasse-Arf theorem ensures that the upper ramification jumps are integer. Moreover, we can find the concrete values of these jumps using results by Brylinski \cite{Br}. 

On the other hand, when the extension is no longer abelian, the ramification jumps are not necessarily integer. Furthermore, to the best of the author's knowledge, the jumps have not been calculated explicitly in most cases. In this paper, we present the values of the upper ramification jumps for some of the finite non-abelian extensions, expecting that our results provide a foothold for obtaining more general conclusions in the future.

If $L/K$ is a totally ramified Artin-Schreier extension, then $L/K$ admits a defining equation of the form $x^p-x=a$ with $a$ being an element of $K$ whose valuation $v(a)$ is negative and prime to $p$. It is well known that the unique jump of the upper ramification groups is located at $-v(a)$.

In this study, we will generalize this relationship between the defining equation and the ramification jumps to any totally ramified finite non-abelian Galois $p$-extensions of $K$ with Galois group of particular structures. Examples of such extensions include, but not limited to, any totally ramified Galois extension of $K$ with Galois group isomorphic to the Heisenberg group over $\mathbb{F}_p$. We will present some conditions on the defining equation under which the ramification jumps can be calculated from the coefficients of the defining equation. We will also show that there always exists a defining equation satisfying the conditions, if the Galois group is isomorphic to $G(n,\mathbb{F}_p)$ defined in Definition \ref{dfb} for some $2\leq n\leq p$.

The group $G(n,\mathbb{F}_p)$ is of maximum nilpotency class, i.e., the descending central series of $G(n,\mathbb{F}_p)$ is the longest among the groups of the same order. Combined with the abelian cases, the calculation of the ramification groups for this case should give us some insight that is useful when calculating the ramification groups of non-abelian Galois extensions in general. This is a motivation for us to consider the Galois extensions with Galois group isomorphic to $G(n,\mathbb{F}_p)$.

Our results are related to a paper \cite{Ab}, in which Abrashkin calculated the upper ramification groups of a profinite extension $K_p/K$ over some local field $K$ of equal characteristic $p$. The extension $K_p/K$ considered in his paper is the composition of all finite Galois $p$-extension over $K$ of nilpotency class less than $p$. In \cite{Ab}, the upper ramification groups were expressed in terms of generators of the Galois group of $K_p/K$. Thus we can compute the ramification groups of some subextension $L/K$ of $K_p/K$ using the results of \cite{Ab}, if we know which generators generate the subgroup $\Gal(K_p/L)$ of $\Gal(K_p/K$).

In this paper, by contrast, we express the ramification jumps of the extension $M_n/K$ defined in Definition \ref{dfd} using the coefficients of a defining equation of $M_n/K$. We evaluate the values of the ramification jumps by using some elementary equalities and inequalities, such as the ones proven in Lemma \ref{lah} and Lemma \ref{laf}.

This paper consists of three parts: Section \ref{sa}, Section \ref{sb}, and Section \ref{sc}. In the first two sections, we give some preliminary results and settings. In the last section, we prove our main theorem, Theorem \ref{ta}.

In Section \ref{sa}, we give some basic results on Galois cohomology required to prove the main theorem.
Using Galois cohomology, we give a combination of $a\in K,b\in K^n$ defining $M_n/K$. At the last of this section, we will give some calculations required to prove the main theorem in the last section. %

In Section \ref{sb}, we introduce filtrations $F_\bullet K,F_\bullet \Omega_{K/k}^1,F_\bullet H^1(K)$ on $K,\Omega_{K/k}^1,\;H^1(K):=H^1(K,\mathbb{F}_p)$, and the graded modules $\Gr K,\Gr \Omega_{K/k}^1,\Gr H^1(K)$ associated to these filtrations as in \cite{AS}. We investigate the relationship between the graded modules defined for $K$ and $L$, where $L/K$ is an Artin-Schreier extension. This is required since $M_n/K(\alpha)$ is a composition of $n$ Artin-Schreier extensions, where $\alpha$ is a solution of $x^p-x=a$.

In Section \ref{sc}, we give our main result, i.e., the calculation of the largest upper ramification jump $r_n$. We reduce the problem to the calculation of the conductor $m'_n$ of the Artin-Schreier extension of $K(\alpha)$ defined by $x^p-x=c_n$ for $c_n\in K(\alpha)$ defined in (\ref{ead}). Since $c_n$ does not belong to $K$, the information of $c_n$ required to calculate $r_n$ has to be written down using the elements of $K$. This is the essential part of this paper, and constitutes the second half of this section.

Take $c'_n,c''_n\in M_n$ such that
\begin{equation}
c_n-{c''_n}^p+c''_n=c'_n,\; -v_{K(\alpha)}(c'_n)=m'_n.
\end{equation}
Let
\begin{equation}
s_n=\max\left(-v_{K(\alpha)}\left(t\frac{dc_n}{dt}\right),-v_{K(\alpha)}\left(t\frac{dc''_n}{dt}\right)\right),
\end{equation}
where $t$ denotes a uniformizer of $K(\alpha)$. First, we express $s_n$ in terms of the coefficients of the defining equation in Corollary \ref{ra}(\ref{ra-2}). The main ingredient of the proof of Theorem \ref{ta} is to prove that the inequality in
\begin{equation}
m'_n=-v_{K(\alpha)}\left(t\frac{dc'_n}{dt}\right)\leq s_n
\end{equation}
is actually an equality, by calculating the image of $c'_n$ in $\Gr_{s_n}H^1(K(\alpha))$ using the results from Section \ref{sb}.

\section{Preliminaries}\label{sa}

Let $B,C$ be groups equipped with the discrete topology, $\Gamma$ a profinite group acting continuously on $B$ and $C$, and $f,g:B\to C$ group homomorphisms preserving the actions of $\Gamma$. Assume that the map $h:B\to C$ defined by $y\mapsto f(y)g(y)^{-1}$ is a surjection. Let $A$ denote the inverse image of $\{1\}$ by $h$. Then $A$ is a subgroup of $B$ with a continuous action of $\Gamma$.

We consider the following sequence:
\begin{equation}
\{1\}\to A\hookrightarrow B\overset{h}{\to} C\to\{1\}.\\
\end{equation}
This is ``exact'' in the sense that the images of the maps coincide with the inverse image of $\{1\}$ by the next map. Furthermore, the map $B\times A\to B\times_CB$ defined by $(y,x)\mapsto (y,yx)$ is a bijection. Nevertheless, this is not an exact sequence, because $h$ is not in general a group homomorphism. However, we can still consider the ``long exact sequence of cohomology'' for this ``exact'' sequence as in \cite{Se}, VII, Annex.\\

\begin{prop}\label{lk}
  \begin{enumerate}[(a)]
    \item\label{lk-1} For $y\in B$ such that $h(y)\in C^\Gamma$, the map $\zeta_y:\Gamma\to B$ defined by $\sigma\mapsto y^{-1}\sigma(y)$ is a 1-cocycle of $A$. Moreover, we can define a map $\delta:C^\Gamma\to H^1(\Gamma,A)$ as follows:
  \begin{equation}
\delta(z)=\bar{\zeta_y},
\end{equation}
where $h(y)=z\in C^\Gamma$, and $\bar{\zeta_y}$ denotes the class of $\zeta_y$ as a 1-cocycle of $A$.
\item\label{lk-3} The image of the map $\delta:C^\Gamma\to H^1(\Gamma,A)$ coincides with the inverse image of $\{1\}$ by $H^1(\Gamma,A)\to H^1(\Gamma,B)$.
\item\label{lk-4} Take $z,z'\in C^\Gamma$. Let $\Gamma_z$ and $\Gamma_{z'}$ denote the intersections of the stabilizer subgroups of $\Gamma$ with respect to the elements in $h^{-1}(\{z\})$ and $h^{-1}(\{z'\})$ respectively. Then $\Gamma_z=\Gamma_{z'}$ if and only if there exists $y\in B^{\Gamma_z\Gamma_{z'}}$ such that $z'=f(y)zg(y)^{-1}$.
\end{enumerate}
\end{prop}

\begin{proof}
(\ref{lk-1}) For any $y\in B$ such that $h(y)\in C^\Gamma$, we have
\begin{equation}
\{yx\in B\ |\ x\in A\}=h^{-1}(\{h(y)\})\supset \{\sigma(y)\ |\ \sigma\in G\}.
\end{equation}
Thus, for any $\sigma\in\Gamma$, we have $y^{-1}\sigma(y)\in A$. Since for any $\sigma,\tau\in\Gamma$,
\begin{equation}
y^{-1}\sigma(y)\sigma(y^{-1}\tau(y))=y^{-1}\sigma\tau(y),
\end{equation}
$\zeta_y$ is a 1-cocycle of $A$. For any $x\in A$, $\zeta_{yx}$ is cohomologous to $\zeta_y$, since
\begin{equation}
\zeta_{yx}(\sigma)=x^{-1}\zeta_y(\sigma)\sigma(x).
\end{equation}
Thus $\delta$ is well-defined.

(\ref{lk-3}) Take a 1-cocycle $s:\Gamma\to A$ of $A$.

Assume that $H^1(\Gamma,A)\to H^1(\Gamma,B)$ sends the class $\bar{s}$ of $s$ to 1. Since $A\subset B$, this implies that there exists $y\in B$ such that $s(\sigma)=y^{-1}\sigma(y)$. Since $s(\sigma)\in A$, we have
\[
h(y)=f(y)g(y)^{-1}=f(y)h(s(\sigma))g(y)^{-1}=f(y)f(y^{-1}\sigma(y))g(y^{-1}\sigma(y))^{-1}g(y)^{-1}
\]
\begin{equation}
=f(\sigma(y))g(\sigma(y))^{-1}=\sigma(f(y)g(y)^{-1})=\sigma(h(y)).
\end{equation} 
Hence $h(y)\in C^\Gamma$ and $\delta(h(y))=\bar{s}$.

Conversely, Take $z\in C^\Gamma$. Let $y$ denote an element of $B$ satisfying $h(y)=z$ and let $s=\zeta_y$. Then we have $s(\sigma)=y^{-1}\sigma(y)$ for all $\sigma\in\Gamma$ and $\delta(z)=\bar{s}$. Since this is a $B$-coboundary, we have that $H^1(\Gamma,A)\to H^1(\Gamma,B)$ sends $\bar{s}$ to 1.

(\ref{lk-4})
Assume $\Gamma_z=\Gamma_{z'}$. Then for all $\sigma\in\Gamma_z\Gamma_{z'}=\Gamma_z=\Gamma_{z'}$, $y_0\in h^{-1}(\{z\})$, and $y_0'\in h^{-1}(\{z'\})$, we have $y_0'y_0^{-1}=\sigma(y_0'y_0^{-1})$ in $B$. Therefore we have $y_0'{y_0}^{-1}\in B^{\Gamma_z\Gamma_{z'}}$ for all $y_0\in h^{-1}(\{z\})$ and $y_0'\in h^{-1}(\{z'\})$. Meanwhile, we have
\[
f(y_0'{y_0}^{-1})zg(y_0'{y_0}^{-1})^{-1}=f(y_0')f(y_0)^{-1}h(y_0)g(y_0)g(y_0')^{-1}
\]
\begin{equation}
=f(y_0')g(y_0')^{-1}=h(y_0')=z'.
\end{equation}
Therefore setting $y=y_0'{y_0}^{-1}$ for some $y_0\in h^{-1}(\{z\})$ and $y_0'\in h^{-1}(\{z'\})$ yields $z'=f(y)zg(y)^{-1}$. 

Conversely, assume that there exists $y\in B^{\Gamma_z\Gamma_{z'}}$ such that $z'=f(y)zg(y)^{-1}$. Take $y_0\in h^{-1}(\{z\})$ and let $y_0'=yy_0$. Then $y_0'\in h^{-1}(\{z'\})$, since $h(y_0')=f(y)h(y_0)g(y)^{-1}=z'$. Since $y$ is fixed by $\Gamma_z\cup\Gamma_{z'}\subset\Gamma_z\Gamma_{z'}$, we have ${y_0'}^{-1}\sigma(y_0')={y_0}^{-1}\sigma(y_0)=1$ for all $\sigma\in\Gamma_z\cup\Gamma_{z'}$. Thus, $\Gamma_z\supset\Gamma_z\cup\Gamma_{z'}$. Hence, $\Gamma_z\supset\Gamma_{z'}$. By symmetry, we also have $\Gamma_{z'}\supset\Gamma_z$. Thus, we have $\Gamma_z=\Gamma_{z'}$.
\end{proof}

\begin{prop}\label{pa}
   Let $G$ be a group. Assume that $G$ admits a descending normal series of subgroups $G=G_0\trianglerightneq G_1\trianglerightneq\cdots\trianglerightneq G_r=\{1\}$. For $0\leq i<r$, let $\pi_i$ denote the canonical projection $G_i\to G_i/G_{i+1}$. Let $f,g:G\to G$ be group homomorphisms satisfying $f^{-1}(G_i)=g^{-1}(G_i)=G_i$ for all $i$. Define a map $h:G\to G$ by $y\mapsto f(y)g(y)^{-1}$, and for all $i$, let $h'_i:G_i/G_{i+1}\to G_i/G_{i+1}$ be the morphism induced by $h$ and $\pi_i$. Then $h$ is surjective if $h'_i$ is surjective for all $i$.
\end{prop}

\begin{proof} We prove this proposition by induction on the length $r$ of the descending normal series of $G$. If $r=1$, then we have $G_1=\{1\}$ and $\pi$ is the identity map. Hence $h$ is surjective if $h'_0$ is surjective.

Suppose $n>1$. Since $G_1$ admits a descending normal series of length $r-1$, it suffices to show that $h$ is surjective if $h|_{G_1}$ and $h'_0$ are surjective.

Suppose $h|_{G_1}$ and $h'_0$ are surjective. Consider the following commutative diagram with exact rows.
\begin{equation}
    \xymatrix{
      \{1\}\ar[r]&G_1\ar@{^{(}->}[r]\ar[d]^{h|_{G_1}}&G\ar@{->>}[r]^{\pi}\ar[d]^{h}&G/G_1\ar[r]\ar[d]^{h'_0}&\{1\}\\
      \{1\}\ar[r]&G_1\ar@{^{(}->}[r]&G\ar@{->>}[r]^{\pi}&G/G_1\ar[r]&\{1\}
    }
    \end{equation}
We will prove that $h$ is surjective by a technique similar to that used to prove the five lemma. Note that we cannot simply apply the five lemma, because $h$ is not in general a homomorphism.

Take $c'\in G$. Since $h'_0$ and $\pi$ are surjective, there exists $c\in G$ satisfying $h'_0(\pi(c))=\pi(c')$. By the commutativity of the diagram, $\pi(f(c)g(c)^{-1})=\pi(c')$. Since $\pi$ is a homomorphism, $\pi(f(c)^{-1}c'g(c))=1$. Then by the exactness of the lower row, we have $f(c)^{-1}c'g(c)\in G_1$. Since $h|_{G_1}$ is surjective, there exists $b\in G_1$ satisfying $h(b)=f(c)^{-1}c'g(c)$. We have
\begin{equation}
h(cb)=f(c)h(b)g(c)^{-1}=f(c)f(c)^{-1}c'g(c)g(c)^{-1}=c'.
\end{equation}
Thus, $h$ is surjective.
\end{proof}

For a field $K$ of characteristic $p>0$, let $K_s$ denote a separable closure of $K$ respectively, and $G_K=\operatorname{Gal}(K_s/K)$ the absolute Galois group of $K$. Let $G$ be a unipotent group over $K$, i.e., an algebraic subgroup of the group of $n\times n$ unitriangular matrices over $K$ for some $n$. 

\begin{defi}
  We say that a unipotent group $G$ over $K$ is split if it admits a finite descending normal series of subgroups whose quotients are isomorphic to the additive group $\mathbb{G}_a$ (cf. \cite{Bo}, Definition 15.1).
\end{defi}

\begin{rema}
  By \cite{Bo}, Theorem 15.5(ii), every connected unipotent group is split if $K$ is perfect. However, this is not true if $K$ is not perfect. According to \cite{Oe}, V.3.4, if $K$ is not perfect and $t$ is an element of $K-K^p$, then the algebraic subgroup $\{(x,y)|x^p-x-ty^p=0\}$ of $\mathbb{G}_a\times\mathbb{G}_a$ is not split.
\end{rema}

\begin{prop}\label{lq}
\begin{enumerate}[(a)]
  \item\label{lq-1} ({\fontencoding{T5}\fontfamily{lmr}\selectfont Nguy\~\ecircumflex{}n}) \cite{Ng} A connected unipotent group $G$ over $K$ is split if and only if $H^1(G_L,G)=\{1\}$ for every extension $L/K$.
  \item\label{lq-2} Let $G$ be a split unipotent group over $\mathbb{F}_p$. Let $F:G\to G$ denote the morphism defined by the absolute Frobenius and let $P:G\to G$ denote the map defined by $y\mapsto F(y)y^{-1}$. Then $P$ is surjective.
 \end{enumerate}
\end{prop}

\begin{proof}
  (\ref{lq-1}) See \cite{Ng}.

(\ref{lq-2}) Since $K_s\to K_s$ defined by $x\mapsto x^p-x$ is surjective, by Proposition \ref{pa}, $P$ is surjective.
\end{proof}
For any positive integer $m$, let us denote the group of $m\times m$ matrices (resp. invertible matrices) over some field of characteristic $p$ by $M(m)$ (resp. $GL(m)$). 
\begin{defi}\label{dfb}
  Fix an integer $2\leq n\leq p$.
  \begin{enumerate}[(a)]
    \item\label{dfb-1} Let $A$ be the nilpotent matrix of size $n\times n$ defined by
\begin{equation}
A=\left(\delta_{i,j-1}\right)_{ij}\in M(n),
\end{equation}
where $\delta$ denotes the Kronecker delta. Let $R(t,x)$ be the $(p-1)$-st Maclaurin polynomial of $(1+t)^x$ with respect to $t$, i.e.,
\begin{equation}
R(t,x)=\sum_{i=0}^{p-1}\binom{x}{i}t^i\in\mathbb{F}_p[t,x].
\end{equation}
Define a morphism $\mathbf{A}:\mathbb{G}_a\to GL(n)$ of algebraic groups by $\mathbf{A}(x)=R(A,x)\in GL(n)$.
\item\label{dfb-2}
Let $G(n)\subset GL(n+1)$ be the unipotent algebraic subgroup of dimension $n+1$ over $\mathbb{F}_p$, defined by
\begin{equation}
G(n)=\left\{
\begin{pmatrix}
  \mathbf{A}(x)&y\\
  0&1
\end{pmatrix}\;
\middle|\;x\in \mathbb{G}_a,\;y\in {\mathbb{G}_a}^n\right\}.
\end{equation}
Let
\begin{equation}
G(n)\supsetneq Z_1G(n)\supsetneq\cdots\supsetneq Z_nG(n)=\{1\}
\end{equation}
be the descending central series of $G(n)$.
\end{enumerate}
\end{defi}

\begin{rema}
We have for all $1\leq j\leq n$, 
\begin{equation}
Z_jG(n)=\left\{
  \begin{pmatrix}
  \mathbf{A}(x)&y\\
  0&1
  \end{pmatrix}\in G(n)\;\middle|
  \;y=\begin{pmatrix}
  y_n\\
  \vdots\\
  y_1
  \end{pmatrix},\;
  x=y_1=\cdots=y_j=0
\right\}.
\end{equation}
Moreover, $G(n)$ is split since $\mathbb{F}_p$ is perfect.
\end{rema}

\begin{defi}\label{dfd}
  Fix an integer $2\leq n\leq p$. Let $K$ be a field. Let $M_n/K$ denote a Galois extension whose Galois group is isomorphic to $G(n,\mathbb{F}_p)$. Let $K\subset M_1\subset\cdots\subset M_n$ be the Galois subextensions of $M_n/K$ corresponding to the descending central series of $G(n,\mathbb{F}_p)$.
  \end{defi}

  For any Galois extension $E/F$, we denote the Galois group $\Gal(E/F)$ by $G_{E/F}$. 

Assume that $K$ is of characteristic $p>0$. We will apply Proposition \ref{lk} to $B=C=G(n,K_s),\;\Gamma=G_K,\;f=F,\;g=\operatorname{id}_{G(n,K_s)}$, where $G_K$ denotes the absolute Galois group over $K$, and $F:G(n,K_s)\to G(n,K_s)$ denotes the Frobenius map.  By Proposition \ref{lq}(\ref{lq-2}), the map $h:G(n,K_s)\to G(n,K_s)$ defined by $y\mapsto F(y)y^{-1}$ is a surjection. Thus we can apply Proposition \ref{lk}.

We have $A=h^{-1}(\{1\})=G(n,\mathbb{F}_p)$. Since $G_K$ acts on $G(n,\mathbb{F}_p)$ trivially, we can identify $H^1(G_K,G(n,\mathbb{F}_p))$ with the set of conjugacy classes of $\operatorname{Hom}(G_K,G(n,\mathbb{F}_p))$ by $G(n,\mathbb{F}_p)$. 
  
\begin{lemm}\label{lag}
  \begin{enumerate}[(a)]
    \item\label{lag-1} There exist $a\in K,\;b\in K^{n}$ such that the extension $M_n/K$ is defined by 
\begin{equation}\label{ea}
\begin{pmatrix}
  \mathbf{A}(x^p)&F(y)\\
  0&1
\end{pmatrix}
=
\begin{pmatrix}
  \mathbf{A}(a)&b\\
  0&1
\end{pmatrix}
\begin{pmatrix}
  \mathbf{A}(x)&y\\
  0&1
\end{pmatrix},
\end{equation}
where $F:{\mathbb{G}_a}^n\to{\mathbb{G}_a}^n$ denotes the component-wise Frobenius map.
\item\label{lag-2} Take
  \begin{equation}
  \alpha\in K_s,\;\gamma=\begin{pmatrix}
    \gamma_n\\
    \vdots\\
    \gamma_1
\end{pmatrix}\in {K_s}^n
  \end{equation}
  such that
\begin{equation}
\alpha^p-\alpha=a,\;F(\gamma)-\gamma=\mathbf{A}(-\alpha^p)b.
    \end{equation}
    Then for all $1\leq j\leq n$, we have $M_j=K(\alpha,\gamma_1,\ldots,\gamma_j)$.
    \end{enumerate}
  \end{lemm}

  \begin{proof}
(\ref{lag-1}) Let $\pi:G_K\to G_{M_n/K}$ denote the canonical projection. Let $\phi:G_{M_n/K}\to G(n,\mathbb{F}_p)$ be any isomorphism. Define $\delta:G(n,K)\to H^1(G_K,G(n,\mathbb{F}_p))$ as in Proposition \ref{lk}. By Proposition \ref{lk}(\ref{lk-3}) and Proposition \ref{lq}(\ref{lq-1}), $\delta$ is a surjection. Thus, there exists
  \begin{equation}
T=\begin{pmatrix}
  \mathbf{A}(a)&b\\
  0&1
\end{pmatrix}
\in G(n,K)
\end{equation}
such that $\delta(T)$ equals the conjugacy class of $\phi\circ\pi$. Then by definition of $\delta$ in Proposition \ref{lk}(\ref{lk-1}), $M_n/K$ is defined by (\ref{ea}).

(\ref{lag-2}) By applying Proposition \ref{lk}(\ref{lk-4}) to $K(\alpha)$, we may replace $T$ with
  \begin{equation}
\begin{pmatrix}
  \mathbf{A}(-\alpha^p)&0\\
  0&1
\end{pmatrix}
\begin{pmatrix}
  \mathbf{A}(a)&b\\
  0&1
\end{pmatrix}
\begin{pmatrix}
  \mathbf{A}(-\alpha)&0\\
  0&1
\end{pmatrix}^{-1}
=\begin{pmatrix}
  \mathbf{A}(0)&\mathbf{A}(-\alpha^p)b\\
  0&1
  \end{pmatrix}
\end{equation}
Thus $M_n=K(\alpha,\gamma_1,\ldots,\gamma_n)$. Note that $G_{M_n/M_j}$ acts trivially on $\alpha,\gamma_1,\ldots,\gamma_j$, but not trivially on $\gamma_{j+1},\ldots,\gamma_n$. Therefore, $M_j=K(\alpha,\gamma_1,\ldots,\gamma_j)$.
  \end{proof}

  \begin{lemm}\label{lah}
  Let $l_1,\ldots,l_n$ and $\lambda_1,\ldots,\lambda_n$ be sequences of integers and $\lambda$ an integer. Assume that for all $2\leq i\leq n$, they satisfy the following conditions:
  \begin{enumerate}[(i)]
  \item\label{lah-2} If $l_i<-(p-1)\lambda+p\lambda_i$, then we have $\lambda_{i-1}=-\lambda+\lambda_i$.
  \item\label{lah-3} We have $l_i\leq -(n-i)\lambda+\lambda_n$.
  \item\label{lah-4} We have $\lambda<\lambda_i$.
  \end{enumerate}

  Then we have
  \begin{equation}\label{eg}
    l_i<-(p-1)\lambda+p\lambda_i,
  \end{equation}
   and
  \begin{equation}\label{eh}
    \lambda_i=(i-1)\lambda+\lambda_1
  \end{equation}
  for all $2\leq i\leq n$.
\end{lemm}

\begin{proof}
  We prove this lemma by induction on $n$. It is clear that the lemma holds for $n=1$.

  Suppose $n=j$. Assume that the lemma holds for $n=j-1$.
  
  By (\ref{lah-3}) for $i=j$, we have
  \begin{equation}
  l_j\leq \lambda_j.
  \end{equation}
  By (\ref{lah-4}) for $i=j$, we have
  \begin{equation}
  \lambda_j<-(p-1)\lambda+p\lambda_j.
  \end{equation}
  Thus we get (\ref{eg}) for $i=j$.
  By (\ref{lah-2}) for $i=j$, we have
  \begin{equation}\label{en}
  -(j-1-i)\lambda+\lambda_{j-1}=-(j-i)\lambda+\lambda_j.
  \end{equation}

  The integer $\lambda$ and the sequences $l_1,\ldots,l_{j-1}$ and $\lambda_1,\ldots,\lambda_{j-1}$ clearly satisfy conditions (\ref{lah-2}) and (\ref{lah-4}) of the lemma for $n=j-1$. We will show that they also satisfy condition (\ref{lah-3}).

  By (\ref{en}) and (\ref{lah-3}) for $2\leq i\leq j-1$, we have
  \begin{equation}
  l_i\leq-(j-1-i)\lambda+\lambda_{j-1}
  \end{equation}
  for all $2\leq j\leq i-1$. Therefore, by induction hypothesis, we have (\ref{eg}) and (\ref{eh}) for all $2\leq j\leq i-1$. By (\ref{eh}) for $i=j-1$ and (\ref{en}), we have (\ref{eh}) for $i=j$.

  Hence, the lemma also holds for $n=j$.
\end{proof}

\begin{lemm}\label{laf}
  We have an equality
  \begin{equation}
  \sum_{i=1}^j\frac{(-1)^{i-1}y}{(j-i)!(i-1)!(x+(i-1)y)}=\prod_{i=1}^{j}\frac{y}{x+(i-1)y}
  \end{equation}
  in $\displaystyle{\mathbb{Z}\left[\frac{1}{(j-1)!},x,y,\frac{1}{\prod_{i=1}^{j}x+(i-1)y}\right]}$.
  \end{lemm}

\begin{proof}
  We may assume $y=1$. Then it follows from the fact that the polynomial of degree $j-1$,
  \begin{equation}
  f(x)=\sum_{i=1}^j\left(\frac{(-1)^{i-1}}{(j-i)!(i-1)!}\prod_{\substack{1\leq i'\leq j\\i'\neq i }}(x+(i'-1))\right)
  \end{equation}
   satisfies $f(0)=f(-1)=\cdots=f(-(j-1))=1$. 
\end{proof}

\section{Filtrations}\label{sb}

Let $K$ be a complete discrete valuation field of characteristic $p>0$ with perfect residue field $k$ and $K_s$ a separable closure of $K$. Define filtrations $F_\bullet K$ of $K$ as $F_nK={\mathfrak{m}}^{-n}$. 

Let $P:K_s\to K_s$ be the surjective map defined by $P(x)=x^p-x$. Then by Proposition \ref{lk}, Proposition \ref{pa} and Proposition \ref{lq}(\ref{lq-1}), we can identify the cokernel of $P|_K:K\to K$ with $H^1(K):=H^1(G_K,\mathbb{F}_p)=\operatorname{Hom}(G_K,\mathbb{F}_p)$ by the following isomorphism:
\begin{equation}
\begin{array}{ccc}
  \operatorname{Coker}P|_K&\to&\operatorname{Hom}(G_K,\mathbb{F}_p)\\
  x&\mapsto&(\sigma\mapsto \sigma(y)-y) 
\end{array}
\end{equation}
where $y\in K_s$ satisfies $P(y)=x$. 

Consider the map $K\to H^1(K)$ defined by the projection $K\to K/P(K)=\operatorname{Coker}P|_K$. Let $F_nH^1(K)$ be the image of $F_n(K)$ by this map in $H^1(K):=H^1(K,\mathbb{F}_p)$.

Let $\Gr_nK=F_nK/F_{n-1}K,\Gr_nH^1(K)=F_nH^1(K)/F_{n-1}H^1(K)$ denote the graded quotients, and define the graded algebra $\Gr K:=\oplus_{n\in\mathbb{Z}}\Gr_nK$. The graded algebra $\Gr K$ is isomorphic to $k[t,t^{-1}]$.

The space $\Omega_{K/k}^1$ of K\"ahler differentials is a 1-dimensional $K$-vector space, and its submodule $\Omega_{\mathcal{O}_K/k}^1$ is a free $\mathcal{O}_K$-module of rank 1. Let $d:K\to\Omega_{K/k}^1$ denote the canonical derivation. Let $F_n\Omega_{K/k}^1=\mathfrak{m}^{-n-1}\Omega_{\mathcal{O}_K/k}^1$ and define the graded quotient $\Gr_n\Omega_{K/k}^1$ and the graded module $\Gr\Omega_{K/k}^1$ as above. Note that this graded module is a $\Gr K$ module. For $\chi\in \Omega_{K/k}^1$, let  $v_K(\chi)$ denote the smallest integer $n$ such that $\chi\in F_{-n}\Omega_{K/k}^1$.\\

\begin{lemm}\label{lr}
\begin{enumerate}[(a)]
 \item\label{lr-1} Let $t$ denote a uniformizer of $K$. The multiplication $K\to\Omega_{K/k}^1$ by $t^{-1}dt$ induces an isomorphism $\mu_{K}:\Gr_nK\to\Gr_n\Omega_{K/k}^1$. This isomorphism does not depend on the uniformizer $t$.
 \item\label{lr-2} The derivation $d$ induces a morphism $\partial:\Gr_nK\to\Gr_n\Omega_{K/k}^1$. We have $\partial=-n\cdot\mu_{K}$.
\end{enumerate}
\end{lemm}
\begin{proof} (\ref{lr-1}) The multiplication by $t^{-1}dt$ induces an isomorphism $\mu_{K}$, since the multiplication is clearly an isomorphism and $F_n\Omega_{K/k}^1=t^{-1}dtF_nK$. For any uniformizer $t,t'$ of $K$, there exist $0\neq a\in k$ and $b\in\mathcal{O}_K$ satisfying $t'=at+bt^2$. Hence we have
  \begin{equation}
  t'^{-1}dt'=t^{-1}(a+bt)^{-1}(adt+ d(bt^2)).
  \end{equation}
  Since $(a+bt)^{-1}\in a^{-1}+\mathfrak{m}$ and $d(bt^2)\in \mathfrak{m}dt$, we have
  \begin{equation}
t'^{-1}dt'\equiv t^{-1}dt \mod \mathcal{O}_Kdt.
\end{equation}
Therefore, $\mu_{K}$ does not depend on the uniformizer.

(\ref{lr-2})  Let $t$ be a uniformizer of $K$. Since $d(t^{-n})=-nt^{-n-1}dt$, $d$ induces $\partial$, and we have $\partial=-n\cdot\mu_{K}$.
\end{proof}

\begin{lemm}\label{lg}
  Let $L/K$ be a finite separable totally ramified extension of complete discrete valuation fields with residue field $k$, $e$ the ramification index of $L/K$,  and $\delta$ the valuation of the different of $L/K$. For any integer $n$, let $n'=e(n+1)-\delta-1$. 

  The canonical morphisms $K\to L$ and $\Omega_{K/k}^1\to\Omega_{L/k}^1$ induce $F_nK\to F_{en}L$, $F_{n-1}K\to F_{en-1}L$, $F_n\Omega_{K/k}^1\to F_{n'}\Omega_{L/k}^1$, and $F_{n-1}\Omega_{K/k}^1\to F_{n'-1}\Omega_{L/k}^1$. There exists a unique non-zero element $\theta\in\Gr_{e-\delta-1}L\simeq k$ such that for all $n$, the diagram below is commutative, where $\mu_{K}$ is the morphism in Lemma \ref{lr}(\ref{lr-1}) and $\mu_{L}:\Gr_nL\to\Gr_n\Omega_{L/k}^1$ is the morphism defined in the same manner.
    \begin{equation}\label{da}
    \xymatrix{
      \Gr_nK\ar[r]\ar[d]_{\mu_{K}}&\Gr_{en}L\ar[r]^{\theta\cdot}&\Gr_{n'}L\ar[d]^{\mu_{L}}\\
      \Gr_n\Omega_{K/k}^1\ar[rr]&&\Gr_{n'}\Omega_{L/k}^1
    }
    \end{equation}
    Moreover, for all $\chi\in\Omega_{K/k}^1$ such that $-v_K(\chi)=n$, we have $-v_L(\chi)=n'=e(n+1)-\delta-1$. 
\end{lemm}

\begin{proof}
  Let $t_K$ and $t_L$ denote a uniformizer of $K$ and $L$ respectively. Let
  \begin{equation}
    \lambda=\frac{{t_K}^{-1}dt_K}{{t_L}^{-1}dt_L}\in F_{e-\delta-1}L.
  \end{equation}
  The following diagram is commutative:
  \begin{equation}
  \xymatrix{
    K\ar[r]^{\lambda\cdot}\ar[d]&L\ar[d]\\
    \Omega_{K/k}^1\ar[r]&\Omega_{L/k}^1,
    }
  \end{equation}
  where the lower horizontal arrow denotes the canonical morphism, and the left and right vertical arrows denote the multiplication by ${t_K}^{-1}dt_K$ and ${t_L}^{-1}dt_L$ respectively. Thus setting $\theta\in\Gr_{e-\delta-1}L$ as the image of $\lambda$ makes the diagram (\ref{da}) commutative. It follows from Lemma \ref{lr}(\ref{lr-1}) that $\theta$ does not depend on the choices of $t_K$ and $t_L$.
  \end{proof}

 \begin{lemm}\label{le}
 Let $n$ be an integer. Consider $\nu:\Gr_nK\to\Gr_nH^1(K)$ induced by the canonical morphism $K\to H^1(K)$. If $n>0$ and $p\nmid n$, the morphism $\nu$ is an isomorphism. If $n=0$, the morphism $\nu$ is a surjection and $\Gr_0H^1(K)$ is isomorphic to $H^1(k)$. Otherwise, the morphism $\nu$ is the zero-map. 
\end{lemm}

\begin{proof}
  Suppose $n>0$ and $p\nmid n$. Since there is no $x\in K$ satisfying $v_K(P(x))=-n$, we have $P(K)\cap(F_nK)\subset F_{n-1}K$. Thus, the morphism $\nu$ is an isomorphism.

Suppose $n=0$. The surjectivity follows from the definition of $\nu$. Since $F_{-1}K\subset P(K)$, we have
\[
\Gr_0H^1(K)\simeq (F_0K/F_{-1}K)/((F_0K\cap P(K))/F_{-1}(K))
\]
\begin{equation}
=k/(P(\mathcal{O}_K)/\mathfrak{m}_K)=k/P(k)\simeq H^1(k).
\end{equation}

Suppose $n>0$ and $p\mid n$. Take $n'$ such that $n=pn'$. For all $x\in F_nK$, there exists $y\in F_{n'}K$ such that $-v_K(x-P(y))<n$. Thus, $\nu$ is the zero-map.

Suppose $n<0$. Since $F_{-1}H^1(K)=0$, $\nu$ is the zero-map.
\end{proof}

For a finite Galois extension $L/K$ such that $K$ is a complete discrete valuation field of characteristic $p>0$, let $G_{L/K}$ denote its Galois group, and $L_{L/K},U_{L/K}$ the sets of indices at jumps of the lower and upper ramification groups of $G_{L/K}$ respectively. For $i\geq -1$, Let $G_{L/K,i},{G_{L/K}}^i$ denote the $i$-th lower and upper ramification group of $G_{L/K}$ respectively. Define the Herbrand function $\psi_{L/K}$ as in \cite{Se}, IV, \S 3.

\begin{lemm}\label{lab}
  Let $K$ be a complete discrete valuation field and $M/K$ a finite Galois extension. Let $L/K$ be a Galois subextension of $M/K$. Let $G=G_{M/K}$, $H=G_{M/L}$, and $\psi=\psi_{L/K}$. Then we have $G^i\cap H=H^{\psi(i)}$ for all $i\geq -1$.
\end{lemm}
\begin{proof} By definitions of the lower and upper ramification groups, we have $G^{i}=G_{\psi_{M/K}(i)}$ and $H^{\psi(i)}=H_{\psi_{M/L}\circ\psi(i)}$. By \cite{Se}, IV, \S 3, Proposition 15, we have $\psi_{M/K}=\psi_{M/L}\circ\psi$. By \cite{Se}, IV, \S 1, Proposition 2, $G_{\psi_{M/K}(i)}\cap H$ coincides with $H_{\psi_{M/K}(i)}$.
\end{proof}

      \begin{prop}\label{ls}
        Let $L/K$ be a ramified Artin-Schreier extension defined by $P(x)=a\;(a\in K-P(K))$. Let $m_a>0$ be the smallest integer such that the image of $a$ is in $F_{m_a}H^1(K)$. 
        \begin{enumerate}[(a)]
          \item\label{ls-4} The valuation of the different of $L/K$ is $(m_a+1)(p-1)$.
          \item\label{ls-1} We have $U_{L/K}=\{m_a\}$.
      \item\label{ls-3} We have
    \begin{equation}
\psi_{L/K}(i)=\max(i,pi-(p-1)m_a).
\end{equation}
\item\label{ls-2} For an integer $n$ such that $p\nmid n$, let $n'=pn-(p-1)m_a$ and $n''=\max(n,n')=\psi_{L/K}(n)$. The canonical morphism $H^1(K)\to H^1(L)$ induces $F_nH^1(K)\to F_{n''}H^1(L)$ and $\Gr_nH^1(K)\to \Gr_{n''}H^1(L)$. Define $\theta$ as in Lemma \ref{lg}. Then the $\mathbb{F}_p$-linear map
    \begin{equation}
    \begin{array}{cccc}
      u:&\Gr_{pn}L&\to&\Gr_{n''}L\\
      &x&\mapsto&\left\{
      \begin{array}{ll}
        \sqrt[p]{x}&(n<m_a)\\
        \sqrt[p]{x}+\theta x&(n=m_a)\\
        \frac{n}{m_a}\theta x&(n>m_a)
      \end{array}
      \right.
    \end{array}
    \end{equation}
    makes the diagram below commutative, where $\nu$ is defined as in Lemma \ref{le}
    \begin{equation} \label{db}
    \xymatrix{
      \Gr_nK\ar[r]\ar[d]_{\nu}&\Gr_{pn}L\ar[r]^{u}&\Gr_{n''}L\ar[d]^{\nu}\\
      \Gr_nH^1(K)\ar[rr]&&\Gr_{n''}H^1(L).
    }
    \end{equation}
    If $n\neq m_a$, then the map $u$ is an isomorphism. If $n=m_a$, then the kernel of $u$ is generated by the image of $a$.
    \end{enumerate}
 \end{prop}
      \begin{proof}
(\ref{ls-4}) The claim follows from the beginning (p.\ 42) of Section b) of \cite{Ha}.
        
        (\ref{ls-1}) By (\ref{ls-4}) and \cite{Se}, IV, \S 1, Proposition 4, we have $U_{L/K}=\{m_a\}$.

(\ref{ls-3}) This follows from the definition of $\psi_{L/K}$.
        
        (\ref{ls-2}) Define $\mu_{K}$ and $\mu_{L}$ as in Lemma \ref{lr}(\ref{lr-1}). Take $y\in K$ such that $v_K(y)=-n$. Let $\eta$ be an element of $K_s$ such that $P(\eta)=y$. When $y\notin P(L)$, applying Lemma \ref{lab} to $L(\eta)/K$ we get $y\in F_{n''}H^1(L)$ by (\ref{ls-1}). When $y\in P(L)$, we have $y\equiv 0\in H^1(L)$. Thus $y\in F_{n''}H^1(L)$ also in this case. Thus the canonical morphism $H^1(K)\to H^1(L)$ induces $F_nH^1(L)\to F_{n''}H^1(L)$ and $\Gr_nH^1(K)\to \Gr_{n''}H^1(L)$.

         Let $\bar{y}$ denote the image of $y$ in $\Gr_nK$. By (\ref{ls-4}) and Lemma \ref{lg}, we have $-v_L(dy)=pn-(p-1)m_a=n'$. Since $k$ is perfect, there exists $s\in L$ such that $v_L(y-s^p)>-np$. We have $v_L(s)=-n$. Let ${s'}=y-s^p$. Since $dy\neq 0$ in $\Omega_{L/k}^1$, we may assume $p\nmid v_L({s'})$. Then we have $v_L({s'})=v_L(d{s'})$. Since $d{s'}=dy$, we have $v_L({s'})=-n'$.

        We have ${s'}+s=y-P(s)\equiv y\in H^1(L)$. We will now write the images of ${s'}$ and $s$ in $\Gr_{n'}L$ and $\Gr_nL$ respectively in terms of $\bar{y}$. By Lemma \ref{lr}(\ref{lr-2}) and Lemma \ref{lg}, the image of $d{s'}=dy$ in $\Gr_{n'}\Omega_{L/k}^1$ equals
        \begin{equation}
        -n'\mu_{K}({s'})=\mu_{L}(-n\theta\bar{y}).
        \end{equation}
        Therefore, the image of ${s'}$ equals $\displaystyle{\frac{n}{n'}\theta\bar{y}}$ in $\Gr_{n'}L$. By definition, the image of $s$ in $\Gr_nL$ is $\sqrt[p]{\bar{y}}$.
Thus the image of ${s'}+s$ equals
  \begin{equation}
  \left\{\begin{array}{ll}
        \sqrt[p]{\bar{y}}&(n<m_a)\\
        \sqrt[p]{\bar{y}}+\theta\bar{y}&(n=m_a)\\
        \frac{n}{n'}\theta\bar{y}&(n>m_a)
      \end{array}\right.
  \end{equation}
  in $\Gr_{n''}L$. Since $n'\equiv m_a\mod p$, the diagram (\ref{db}) is commutative.

  If $n<m_a$, then the map $u$ is an isomorphism, since $k$ is perfect.

  If $n>m_a$, then the map $u$ is an isomorphism, since $n$ is prime to $p$.

  If $n=m_a$, then the kernel of $u$ is generated by $(-\theta)^{-\frac{p}{p-1}}$. On the other hand, the image of $a$ in $H^1(L)$ equals 0. Thus, there exists $i\in\mathbb{F}_p^\times$ such that $a\equiv i(-\theta)^{-\frac{p}{p-1}}$ in $\Gr_{pm_a}L$, and the kernel of $u$ is generated by the class of $a$.
\end{proof}

\section{Calculation of the Ramification Groups}\label{sc}
Let $2\leq n\leq p$. Recall the algebraic group $G(n)\subset GL(n+1)$ over $\mathbb{F}_p$ and its descending central series 
\begin{equation}
G(n)\supsetneq Z_1G(n)\supsetneq\cdots\supsetneq Z_nG(n)=\{1\}
\end{equation}
of Definition \ref{dfb}(\ref{dfb-2}).  

Let $K$ be a complete discrete valuation field, and $K_s$ a separable closure of $K$. Define $P:K_s\to K_s$ by $x\mapsto x^p-x$.  Assume that the residue field $k$ of $K$ is perfect of characteristic $p>0$. Take $K\subset M_1\subset\cdots\subset M_n$ as in Definition \ref{dfd}. Assume that $M_n/K$ is totally ramified. Recall that $L_{E/K},U_{E/K}$ denote the sets of indices at jumps of the lower and upper ramification groups respectively of the Galois group $G_{E/K}$ for a field extension $E/K$. Then we have
\begin{equation}
U_{M_1/K}\subset\cdots\subset U_{M_n/K}
\end{equation}
by \cite{Se}, IV, \S 3, Proposition 14.

\begin{lemm}\label{lad}
  There exists a sequence $r_1<\cdots<r_n$ of rational numbers such that for all $q>r_1$, we have
  \begin{equation}\label{eo}
  G_{M_n/K}^{q}=G_{M_n/M_j}\simeq  Z_jG(n,\mathbb{F}_p),
  \end{equation}
  where $j$ is the largest integer satisfying $q>r_j$. Moreover, for $1\leq j\leq n$, $r_j$ is the largest element of $U_{M_j/K}$.
\end{lemm}

\begin{proof}
  We prove this lemma by descending induction on $j$.

  Suppose $j=n$. Let $r_n$ be the largest element of $U_{M_{n}/K}$. Then we have (\ref{eo}) for $q>r_n$.
  
  Suppose $1\leq j\leq n-1$. Assume that we have (\ref{eo}) for $q>r_{j+1}$. Let $l,r_j$ be the largest elements of $L_{M_{j+1}/K},U_{M_{j}/K}$ respectively. Then we have $G_{M_n/K}^{q}\subset G_{M_n/M_j}$ for $q>r_j$ by \cite{Se}, IV, \S 3, Proposition 14. It suffices to show $G_{M_{j+1}/K}^{r_{j+1}}=G_{M_{j+1}/K,l}=Z(G_{M_{j+1}/K})$, since $Z(G_{M_{j+1}/K})=G_{M_n/M_{j}}/G_{M_n/M_{j+1}}$. By \cite{Se}, IV, \S 2, Proposition 10, we have
  \begin{equation}
    [G_{M_{j+1}/K},G_{M_{j+1}/K,l}]=[G_{M_{j+1}/K,1},G_{M_{j+1}/K,l}]\subset G_{M_{j+1}/K,l+2}=\{1\},
  \end{equation}
  since $l$ is the largest element of $L_{M_{j+1}/K}$. Thus, $G_{M_{j+1}/K,l}\subset Z(G_{M_{j+1}/K})$. Since $G_{M_{j+1}/K,l}$ is not trivial and $Z(G_{M_{j+1}/K})$ is isomorphic to $\mathbb{F}_p$, we have $G_{M_{j+1}/K,l}=Z(G_{M_{j+1}/K})$.
\end{proof}

Recall $A\in M(n),\;R(t,x)\in\mathbb{F}_p[t,x],\;\mathbf{A}:\mathbb{G}_a\to GL(n)$ of Definition \ref{dfb}(\ref{dfb-1}).

\begin{lemm}\label{li}
Assume that $K$ is of characteristic $p>0$ and $M_n/K$ is totally ramified. There exist $a\in K,\;b\in K^n$ satisfying the conditions of Lemma \ref{lag}(\ref{lag-1}) and conditions (\ref{li-2})--(\ref{li-3}) below. Let
\begin{equation}\label{ev}
b=
\begin{pmatrix}
  b_n\\
  \vdots\\
  b_1
\end{pmatrix},
\end{equation}
and let $m_a,m_j\;(1\leq j\leq n)$ denote $-v_K(a), -v_K(b_j)$ respectively.  
\begin{enumerate}[(i)]
\item\label{li-2} $m_a,m_1$ are positive and prime to $p$.
\item\label{li-1} For all $2\leq j\leq n$, we have $p\nmid m_j$ if $m_j>0$.
\item\label{li-3} If $n\leq p-1$ and $m_a=m_1$, then the images $\bar{a},\bar{b}_1$ of $a,b_1$ in $\Gr_{m_a}K=\Gr_{m_1}K$ respectively are linearly independent over $\mathbb{F}_p$.
  \end{enumerate}
\end{lemm}

\begin{proof}
  Take $a\in K,\;b\in K^n$ as in Lemma \ref{lag}(\ref{lag-1}).
 By Proposition \ref{lk}(\ref{lk-4}), we may replace
  \[
\begin{pmatrix}
  \mathbf{A}(a)&b\\
  0&1
\end{pmatrix}
\]
by
\[
\begin{pmatrix}
  \mathbf{A}(s^p)&F(t)\\
  0&1
\end{pmatrix}
\begin{pmatrix}
  \mathbf{A}(a)&b\\
  0&1
\end{pmatrix}
\begin{pmatrix}
  \mathbf{A}(s)&t\\
  0&1
\end{pmatrix}^{-1}
\]
\begin{equation}
=
\begin{pmatrix}
  \mathbf{A}(a+P(s))&\mathbf{A}(s^p)b-\mathbf{A}(a+P(s))t+F(t)\\
  0&1
\end{pmatrix}
\end{equation}
for $s\in K,\;t\in {M_n}^n$, if $\mathbf{A}(s^p)b-\mathbf{A}(a+P(s))t+F(t)\in K^n$.

When $n\leq p-1$, let
\begin{equation}
  S(t,x)=\frac{R(t,x)-1}{t}=\sum_{i=0}^{p-2}\binom{x}{i+1}t^i\in\mathbb{F}_p[t,x],
  \end{equation}
  and define a morphism $\mathbf{v}:\mathbb{G}_a\to {\mathbb{G}_a}^n$ of algebraic varieties by
\begin{equation}
\mathbf{v}(x)=S(A,x)
\begin{pmatrix}
  0\\
  \vdots\\
  0\\
  1
\end{pmatrix}\in {\mathbb{G}_a}^n.
\end{equation}
For any $x,y\in\mathbb{G}_a$, we have 
\begin{equation}
  R(A,x+y)=R(A,x)R(A,y)
\end{equation}
by the Chu--Vandermonde identity and $A^p=0$.
Thus, the morphism $\mathbf{v}$ satisfies
  \begin{equation}
  \mathbf{v}(x+y)=\mathbf{v}(x)+\mathbf{A}(x)\mathbf{v}(y).
  \end{equation}
  In particular, we have
  \begin{equation}\label{ex}
  \mathbf{v}(x^p)-\mathbf{A}(P(x))\mathbf{v}(x)=F(\mathbf{v}(x))-\mathbf{A}(P(x))\mathbf{v}(x)=\mathbf{v}(P(x)).
  \end{equation}

  We have the following fact.
  
  \begin{enumerate}[\textbf{Fact (*)}]
  \item For any
  \begin{equation}
t=\begin{pmatrix}
t_n\\
\vdots\\
t_1
\end{pmatrix}\in K^n,
\end{equation}
  the first component from the bottom of $(1-\mathbf{A}(a))t$ is $0$, and the $j$-th component from the bottom of $(1-\mathbf{A}(a))t$ depends only on $t_1,\ldots,t_{j-1}$ for all $2\leq j\leq n$.
  \end{enumerate}
  We may define operations (\ref{li-op-1}), (\ref{li-op-2}) and (\ref{li-op-3}) as follows.
\begin{enumerate}[(I)]
\item\label{li-op-1} Choose appropriate $s\in K$, set $t=0$, and replace $(a,b)$ by $(a+P(s), \mathbf{A}({s}^p)b)$, so that we have $p\nmid m_a$ if $m_a>0$. We know by Lemma \ref{le} that such $s$ exists.
\item\label{li-op-2} Choose appropriate $t_j\in K\;(1\leq j\leq n)$ successively, set $s=0$, and replace $(a,b)$ by $(a,b+(1-\mathbf{A}(a))t+(F(t)-t))$, so that we have $p\nmid m_j$ if $m_j>0$ for all $1\leq j\leq n$. We know by Lemma \ref{le} and Fact (*) that such $t_j\;(1\leq j\leq n)$ exist.
\item\label{li-op-3} Assume $n\leq p-1$, $m_a=m_1$, and $i=\bar{a}^{-1}\bar{b}_1\in\mathbb{F}_p$. Take $\alpha\in M_n$ such that $\alpha^p-\alpha=a$ as in Lemma \ref{lag}(\ref{lag-2}). Set $s=0,\;t=i\mathbf{v}(\alpha)$, and replace $(a,b)$ by $(a,b-i\mathbf{v}(a))$, so that we have $m_a> m_1$. This is valid since we have
  \begin{equation}
    -\mathbf{A}(a)\mathbf{v}(\alpha)+F(\mathbf{v}(\alpha))=\mathbf{v}(a)
  \end{equation}
  by (\ref{ex}).
\end{enumerate}

Define $\alpha, \gamma_j\in M_n\;(1\leq j\leq n)$ as in Lemma \ref{lag}(\ref{lag-2}). Since $M_n$ is totally ramified, $K(\alpha),K(\gamma_1)$ are ramified. By Proposition \ref{ls}(\ref{ls-1}), we have $m_a,m_1>0$.

Performing (\ref{li-op-1}) and (\ref{li-op-2}) successively, conditions (\ref{li-2}) and (\ref{li-1}) are satisfied.

Suppose $n\leq p-1$, $m_a=m_1$ and $\bar{a}^{-1}\bar{b}_1\in\mathbb{F}_p$. Performing (\ref{li-op-3}), condition (\ref{li-3}) is satisfied. However, since (\ref{li-op-3}) may change $m_j\;(1\leq j\leq n)$, we have to perform (\ref{li-op-2}) again to ensure conditions (\ref{li-2}) and (\ref{li-1}) are satisfied. Since (\ref{li-op-2}) does not make $m_1$ larger, condition (\ref{li-3}) remains satisfied.  
\end{proof}

Take $a\in K,\;b\in K^{n}$ satisfying the conditions of Lemma \ref{li}. Define $r_j\;(1\leq j\leq n)$ as in Lemma \ref{lad}.

 Let
  \begin{equation}\label{ej}
 \omega=
  \begin{pmatrix}
    {\omega_n}\\
    \vdots\\
    {\omega_1}
  \end{pmatrix}=\mathbf{A}(-a)db\in (\Omega_{K/k}^1)^n,
  \end{equation}
  where $db$ denotes the component-wise derivation.

  We will now state our main theorem.

  \begin{theo}\label{ta}
    Let $K$ be a complete discrete valuation field of equal characteristic $p>0$. Assume that the residue field $k$ of $K$ is perfect. Let $v_K$ denote the valuation of $K$ and $\Omega_{K/k}^1$ defined at the beginning of Section \ref{sb}. Take $K\subset M_1\subset\cdots\subset M_n$ as in Definition \ref{dfd}. Assume that $M_n/K$ is totally ramified. Take $a\in K,\;b\in K^n$ satisfying the conditions of Lemma \ref{li}. Define $b_1,\ldots,b_n\in K$ as in (\ref{ev}). Let $m_a=-v_K(a),\;m_j=-v_K(b_j)\;(1\leq j\leq n)$. Define the sequence $m_a\leq r_1<\cdots<r_n$ as in Lemma \ref{lad}, and $\omega_j\;(1\leq j\leq n)$ as in (\ref{ej}). Then we have
  \begin{equation}\label{ez}
  r_{j}=\max\left(\max_{1\leq i\leq j}\left(\frac{j-i}{p}m_a-v_K\left(\omega_i\right)\right),\frac{(j+p-2)m_a+m_1}{p}\right)
  \end{equation}
  for all $2\leq j\leq n$.
  \end{theo}

  We will prove this theorem at the end of this paper.

   Define $\alpha,\gamma_j\;(1\leq j\leq n)$ as in Lemma \ref{lag}(\ref{lag-2}) and define $L$ by $L=K(\alpha)$. Let
  \begin{equation}\label{ead}
    c=\begin{pmatrix}
    c_n\\
    \vdots\\
    c_1
    \end{pmatrix}
    =\mathbf{A}(-\alpha^p)b,
  \end{equation}
  and let $m'_j\;(2\leq j\leq n)$ denote the smallest integer such that $c_j\in L$ defines an element of $F_{m'_j}H^1(L)$. Then we have $P(\gamma_j)=c_j$ and $M_j=K(\alpha,\gamma_1,\ldots,\gamma_j)$ for all $1\leq j\leq n$ by Lemma \ref{lag}(\ref{lag-2}).

  \begin{lemm}\label{lac}
  \begin{enumerate}[(a)]
    \item\label{lac-1} Let $2\leq j\leq n$. We have
  \begin{equation}\label{y}
  -v_L(c_j)\leq p\max_{1\leq i\leq j}((j-i)m_a+m_i)
  \end{equation}
  and
  \begin{equation}\label{z}
  -v_L(dc_j)=\max_{1\leq i\leq j}((j-i)m_a-v_L(\omega_i)).
  \end{equation}
  Furthermore, there exists a unique $1\leq i\leq j$ satisfying
  \begin{equation}\label{za}
    -v_L(dc_j)=(j-i)m_a-v_L(\omega_i).
  \end{equation}
  For this unique $i$, the image of $dc_j$ in $\Gr_{-v_L(dc_j)}\Omega_{L/k}^1$ equals that of $\displaystyle{\binom{-\alpha}{j-i}\omega_i}$.
\item\label{lac-2} We have $p\nmid m'_j=\psi_{L/K}(r_j)$ for $2\leq j\leq n$, and $m_a<m'_2<\cdots<m'_n$.
  \end{enumerate}
\end{lemm}

\begin{proof}
  (\ref{lac-1}) Since
$d(a+\alpha)=d(\alpha^p)=0$ in $\Omega_{L/k}^1$, we have
\begin{equation}
c=\mathbf{A}(-a-\alpha)b,\;dc=\mathbf{A}(-a-\alpha)db=\mathbf{A}(-\alpha)\omega.
\end{equation}
Thus we have
\begin{equation}\label{x}
  c_j=\sum_{i=1}^j\binom{-a-\alpha}{j-i}b_i
\end{equation}
and
\begin{equation}\label{u}
dc_j=\sum_{i=1}^j\binom{-\alpha}{j-i}\omega_i.
\end{equation}

We get (\ref{y}) from (\ref{x}). By Lemma \ref{lg} and Proposition \ref{ls}(\ref{ls-4}), we have
\[
-v_L\left(\binom{-\alpha}{j-i}\omega_i\right)=(j-i)m_a-v_L(\omega_i)
\]
\begin{equation}
=(j-i)m_a-pv_K(\omega_i)-(p-1)m_a\equiv(j-i+1)m_a\mod p.
\end{equation}
Therefore, the valuations of the terms in the right-hand side of the equation (\ref{u}) do not coincide with each other. Thus we have (\ref{z}) and the rest of the claim.

(\ref{lac-2}) By Lemma \ref{lab}, we have ${G_{M_j/L}}^{\psi_{L/K}(i)}={G_{M_j/K}}^{i}\cap G_{M_j/L}$ for all $i\geq -1$. Thus by Lemma \ref{lad}, $U_{M_j/L}=\{\psi_{L/K}(r_j)\}\sqcup U_{M_{j-1}/L}$. Note that $\psi_{L/K}(r_j)$ is larger than all of the elements of $U_{M_{j-1}/L}$ and that $M_j=M_{j-1}(\gamma_j)$. Hence $m'_j=\psi_{L/K}(r_j)$. We get $r_1\geq m_a$ from $U_{M_1/K}\supset U_{L/K}=\{m_a\}$ and Lemma \ref{lad}. Since we have $r_1<r_2<\cdots<r_n$ by Lemma \ref{lad}, we have $m_a<m'_2<\cdots<m'_n$. Since $c_j\notin P(L)$, by definition of $m'_j$, we have $p\nmid m'_j>0$.
\end{proof}

Take $c'_j,c''_j\in L\;(2\leq j\leq n)$ such that
\begin{equation}\label{eaa}
c_j-P(c''_j)=c'_j,\; -v_L(c'_j)=m'_j.
\end{equation}
Then we have $dc_j+dc''_j=dc'_j$, since $d({c''_j}^p)=0$.

\begin{lemm}\label{laee}
  We have
  \begin{equation}\label{el}
  -v_L(dc''_j)\leq\max_{1\leq i\leq j}((j-i)m_a+m_{i}).
  \end{equation}
\end{lemm}

\begin{proof}
  By definition of $c''_j$ (\ref{eaa}), we have
  \begin{equation}
  -v_L(dc''_j)\leq-v_L(c''_j)\leq-\frac{1}{p}v_L(c_j).
  \end{equation}
  By (\ref{y}), we have (\ref{el}).  
\end{proof}

Let $\bar{a},\bar{b}_j\;(1\leq j\leq n)$ denote the image of $a,b_j$ in $\Gr_{m_a}K,\Gr_{m_j}K$ respectively, when $m_j\neq-\infty$.

\begin{prop}\label{lae}
  Suppose we have
  \begin{equation}\label{et}
  -v_L(dc_n)\leq \max_{1\leq i\leq n}((n-i)m_a+m_{i}).
  \end{equation}
  \begin{enumerate}[(a)]
    \item\label{lae-1} For all $2\leq j\leq n$, We have
  \begin{equation}\label{aee}
    -v_K(\omega_j)<(j-1)m_a+m_1
  \end{equation}
\item\label{lae-2} For all $1\leq j\leq n$, we have
  \begin{equation}\label{ei}
  p\nmid m_j=(j-1)m_a+m_1>0,
\end{equation}
and
\begin{equation}\label{eq}
\bar{b}_j=\frac{m_1}{(j-1)!m_j}\bar{a}^{j-1}\bar{b}_1.
\end{equation}
\item\label{lae-5} We have $n\leq p-1$.
\item\label{lae-3} For all $1\leq j\leq n$, we have
  \begin{equation}\label{er}
  -v_L(c_j)=pm_j=p(j-1)m_a+pm_1,
  \end{equation}
  and the image of $c_j$ in $\Gr_{pm_j}L$ equals $\displaystyle{\frac{(-m_a\bar{a})^{j-1}\bar{b}_1}{\prod_{i=2}^jm_i}\neq 0}$.
\item\label{lae-4}
  For all $1\leq j\leq n$, we have
  \begin{equation}\label{es}
  -v_L(c''_j)=-v_L(dc''_j)=m_j=(j-1)m_a+m_1
  \end{equation}
  and the image of $c''_j$ in $\Gr_{m_j}L$ equals $\displaystyle{\frac{\left(-m_a\sqrt[p]{\bar{a}}\right)^{j-1}\sqrt[p]{\bar{b}_1}}{\prod_{i=2}^jm_i}}$.
  \end{enumerate}
\end{prop}

\begin{proof}
  (\ref{lae-1}) By Lemma \ref{lg} and Proposition \ref{ls}(\ref{ls-4}), we have $-v_L(\chi)=-pv_K(\chi)-(p-1)m_a$ for all $\chi\in\Omega_{K/k}^1$. Hence it suffices to show that for all $2\leq j\leq n$, we have
  \begin{equation}\label{eac}
    -v_L(\omega_j)<p((j-1)m_a+m_1)-(p-1)m_a.
  \end{equation}
  We will apply Lemma \ref{lah} to
  \begin{equation}
  \lambda=m_a,\;\lambda_j=\max_{1\leq i\leq j}((j-i)m_a+m_{i}),\; l_j=-v_L(\omega_j).
  \end{equation}
  
  By definition of $\omega$ (\ref{ej}), we have
  \begin{equation}\label{eab}
  \omega_j=\sum_{i=1}^j\binom{-a}{j-i}db_i.
  \end{equation}
  We have
  \begin{equation}
    -v_L\left(\binom{-a}{j-i}db_i\right)=p((j-i)m_a+m_{i})-(p-1)m_a
  \end{equation}
   by Lemma \ref{lg} and Proposition \ref{ls}(\ref{ls-4}). By (\ref{eab}) and the property of valuation, condition (\ref{lah-2}) of Lemma \ref{lah} is satisfied.
  
  By (\ref{z}) and (\ref{et}), we get
  \begin{equation}
  (n-j)m_a-v_L(\omega_j)\leq\max_{1\leq i\leq n}((n-i)m_a+m_{i}).
  \end{equation}
  Thus condition (\ref{lah-3}) of Lemma \ref{lah} is satisfied.

  Since $m_a,m_1>0$, we have 
  \begin{equation}
  m_a<(j-1)m_a+m_1\leq\max_{1\leq i\leq j}((j-i)m_a+m_{i})
  \end{equation}
  for all $2\leq j\leq n$. Thus condition (\ref{lah-4}) of Lemma \ref{lah} is satisfied.

  By applying Lemma \ref{lah}, we obtain
  \begin{equation}
  -v_L(\omega_j)< p\max_{1\leq i\leq j}((j-i)m_a+m_i)-(p-1)m_a
  \end{equation}
  and 
  \begin{equation}
    \max_{1\leq i\leq j}((j-i)m_a+m_i)=(j-1)m_a+m_1.
  \end{equation}
  Therefore, we get (\ref{eac}) for all $2\leq j\leq n$.
  
  (\ref{lae-2}) Since $db=\mathbf{A}(a)\omega$, we have 
  \begin{equation}
  -v_K(db_j)=m_j\leq\max_{1\leq i\leq j}((j-i)m_a-v_K(\omega_i)).
  \end{equation}
  By (\ref{aee}) for $2\leq i\leq j$, we have
  \begin{equation}
  (j-i)m_a-v_K(\omega_i)<(j-1)m_a+m_1
  \end{equation}
  for $2\leq i\leq j$. On the other hand, we have $\omega_1=db_1$. Thus, we have
  \begin{equation}
  (j-1)m_a-v_K(\omega_1)=(j-1)m_a+m_1.
  \end{equation}
  Thus we have
  \begin{equation}
    -v_K(db_j)=m_j=\max_{1\leq i\leq j}((j-i)m_a-v_K(\omega_i))=(j-1)m_a+m_1>0
  \end{equation}
  by the property of valuation. By conditions (\ref{li-2}) and (\ref{li-1}) of Lemma \ref{li}, we have (\ref{ei}) for all $1\leq j\leq n$. The image of $db_j$ equals the image of $\displaystyle{\binom{a}{j-1}db_1}$ in $\Gr_{m_j}\Omega_{K/k}^1$. Hence we have (\ref{eq}).

  (\ref{lae-5}) By (\ref{lae-2}), we have (\ref{ei}) for all $2\leq j\leq n$. Since $(j-1)m_a+m_1\mod p\;(1\leq j\leq n)$ are different from each other and $m_1,\ldots,m_n$ are all prime to $p$ by (\ref{ei}), we have $n\leq p-1$.
  
  (\ref{lae-3}) By (\ref{y}) and (\ref{ei}), we have $-v_L(c_j)\leq pm_j$.  By (\ref{x}), Lemma \ref{laf} and (\ref{eq}), the image of $c_j$ in $\Gr_{pm_j}L$ equals
  \begin{equation}
  \sum_{i=1}^j\frac{(-\bar{a})^{j-i}}{(j-i)!}\bar{b}_i=\sum_{i=1}^j\frac{(-\bar{a})^{j-i}m_1}{(j-i)!(i-1)!m_i}\bar{a}^{i-1}\bar{b}_1
  =\frac{(-m_a\bar{a})^{j-1}\bar{b}_1}{\prod_{i=2}^jm_i}\neq 0.
  \end{equation}
  Thus we have (\ref{er}).

  (\ref{lae-4}) Since we have (\ref{er}) and $m_j$ is prime to $p$ by (\ref{ei}), we have (\ref{es}). Hence, the image of $c''_j$ in $\Gr_{m_j}L$ equals $\displaystyle{\frac{\left(-m_a\sqrt[p]{\bar{a}}\right)^{j-1}\sqrt[p]{\bar{b}_1}}{\prod_{i=2}^jm_i}}$, since we have $\sqrt[p]{x}=x$ for $x\in\mathbb{F}_p$.
\end{proof}

We will now express $\max(-v_L(dc_n),-v_L(dc''_n))$ in terms of $a,b$.

\begin{coro}\label{ra}
  \begin{enumerate}[(a)]
    \item\label{ra-1} The following conditions are equivalent:
  \begin{enumerate}[(i)]
  \item\label{ra-c-1} $-v_L(dc_n)<-v_L(dc''_n)$.
  \item\label{ra-c-2} $-v_L(dc_n)<\max_{1\leq j\leq n}((n-j)m_a+m_j)$.
  \item\label{ra-c-3} $-v_L(dc_n)<(n-1)m_a+m_1$.
  \end{enumerate}
\item\label{ra-2} We have
  \[
  m'_n=-v_L(dc'_n)\leq \max(-v_L(dc_n),-v_L(dc''_n))
  \]
  \begin{equation}\label{eu}
  =\max\left(\max_{1\leq i\leq n}\left((n-i-p+1)m_a-pv_K\left(\omega_i\right)\right),(n-1)m_a+m_1\right).
  \end{equation}
  \end{enumerate}
\end{coro}

\begin{proof}
  (\ref{ra-1}) $\text{(\ref{ra-c-1})}\Rightarrow\text{(\ref{ra-c-2})}$: This follows from (\ref{el}).

  $\text{(\ref{ra-c-3})}\Rightarrow\text{(\ref{ra-c-2})}$: This clearly holds.

  $\text{(\ref{ra-c-2})}\Rightarrow\text{(\ref{ra-c-1}),(\ref{ra-c-3})}$: This follows from (\ref{es}).

  (\ref{ra-2}) 
  Since $dc_n+dc''_n=dc'_n$, we have $m'_n\leq\max(-v_L(dc_n),-v_L(dc''_n))$. By (\ref{ra-1}), (\ref{z}), and (\ref{es}), we have (\ref{eu}).
\end{proof}

We will now prove our main theorem.

\begin{proof}[Proof of Theorem \ref{ta}]
  It suffices to show the case where $j=n$, since the case where $j<n-1$ can be reduced to the case $j=n$ by replacing $n$ by $j$.
  
Let
\begin{equation}
  s_n=\psi_{L/K}\left(\max\left(\max_{1\leq i\leq n}\left(\frac{n-i}{p}m_a-v_K\left(\omega_i\right)\right),\frac{(n+p-2)m_a+m_1}{p}\right)\right).
\end{equation}
It suffices to show $m'_n=\psi_{L/K}(r_n)=s_n$ to complete the proof, since $\psi_{L/K}$ is injective by Proposition \ref{ls}(\ref{ls-3}).

Since $\displaystyle{\frac{(n+p-2)m_a+m_1}{p}>m_a}$, we have
\begin{equation}
\max\left(\max_{1\leq i\leq n}\left(\frac{n-i}{p}m_a-v_K\left(\omega_i\right)\right),\frac{(n+p-2)m_a+m_1}{p}\right)>m_a.
\end{equation}
By Proposition \ref{ls}(\ref{ls-3}), we have $\psi_{L/K}(x)=px-(p-1)m_a$ if $x\geq m_a$. Then by Corollary \ref{ra}(\ref{ra-2}), we get
\begin{equation}
s_n=\max(-v_L(dc_n),\;-v_L(dc''_n))\geq m'_n.
\end{equation}
We have $m'_n=s_n$ when $-v_L(dc_n)\neq -v_L(dc''_n)$. It suffices to show that we have $m'_n=s_n$ also when $-v_L(dc_n)=-v_L(dc''_n)$.

Assume $-v_L(dc_n)=-v_L(dc''_n)$. Then by (\ref{el}), the hypothesis (\ref{et}) of Proposition \ref{lae} is satisfied. By Proposition \ref{lae}(\ref{lae-5}), we have $n\leq p-1$. By Lemma \ref{lg} and Proposition \ref{ls}(\ref{ls-4}), we have $-v_L(\omega_i)=-pv_K(\omega_i)-(p-1)m_a\equiv m_a\mod p$ for all $1\leq i\leq n$. Thus, by Lemma \ref{lac}(\ref{lac-1}), there exists a unique integer $1\leq j\leq n$ satisfying $-v_L(dc_n)=(n-j)m_a-v_L(\omega_j)\equiv (n-j+1)m_a\mod p$. Combining with (\ref{es}), we get
\begin{equation}
v_L(dc_n)-v_L(dc''_n)=0\equiv (j-2)m_a+m_1\mod p.
\end{equation}
Since $m_i=(i-1)m_a+m_1$ is prime to $p$ for all $1\leq i\leq n$ by (\ref{ei}), we have $j\neq 2,\ldots,n+1$. Therefore, since $j$ is an integer satisfying $1\leq j\leq n$, we have $j=1$. Hence, we have
\begin{equation}\label{ew}
m_n=(n-1)m_a+m_1=-v_L(dc''_n)=-v_L(dc_n)=(n-1)m_a-v_L(\omega_1).
\end{equation}
Therefore, we have $m_1=-v_L(\omega_1)$. By definition, we have $\omega_1=db_1$. Thus we have $m_1=pm_1-(p-1)m_a$ by Lemma \ref{lg} and Proposition \ref{ls}(\ref{ls-4}). Hence we have $m_a=m_1$. By (\ref{ei}) we get $m_j=jm_a$ for all $1\leq j\leq n$. Hence we have $s_n=-v_L(dc_n)=-v_L(dc''_n)=m_n$.

We have only to check that the image of $dc_n+dc''_n=dc'_n$ in $\Gr_{m_n}\Omega_{L/k}^1$ does not vanish. Define the $k$-linear isomorphism $\mu_{L}:\Gr L\to\Gr \Omega_{L/k}^1$ as in Lemma \ref{lr}(\ref{lr-1}) and $\theta\in \Gr_{-(p-1)m_a}L$ as in Lemma \ref{lg}. Since we have (\ref{ew}), the image of $dc_n$ in $\Gr_{m_n}\Omega_{L/k}^1$ equals that of $\displaystyle{\binom{-\alpha}{n-1}\omega_1}$ by Lemma \ref{lac}(\ref{lac-1}). Thus, by Lemma \ref{lr}(\ref{lr-2}), the image of $dc_n$ in $\Gr_{m_n}\Omega_{L/k}^1$ equals $\displaystyle{-m_1\frac{(-\sqrt[p]{\bar{a}})^{n-1}}{(n-1)!}\mu_{L}(\theta\bar{b}_1)}$.

Meanwhile, by Proposition \ref{lae}(\ref{lae-4}), Lemma \ref{lr}(\ref{lr-2}), and $m_j=jm_a=jm_1$ for all $1\leq j\leq n$, the image of $dc''_n$ in $\Gr_{m_n}\Omega_{L/k}^1$ equals $\displaystyle{-m_1\frac{(-\sqrt[p]{\bar{a}})^{n-1}}{(n-1)!}\mu_{L}\left(\sqrt[p]{\bar{b}_1}\right)}$.

Define $u:\Gr_{pm_a}L\to\Gr_{m_a}L$ as in Proposition \ref{ls}(\ref{ls-2}). Then the image of $dc_n+dc''_n=dc'_n$ in $\Gr_{m_n}\Omega_{L/k}^1$ equals $\displaystyle{-m_1\frac{(-\sqrt[p]{\bar{a}})^{n-1}}{(n-1)!}}$ times
\begin{equation}
\mu_{L}\left(\theta\bar{b}_1+\sqrt[p]{\bar{b}_1}\right)=\mu_{L}(u(\bar{b}_1))
\end{equation}
by Proposition \ref{ls}(\ref{ls-2}). Since $n\leq p-1$, by condition (\ref{li-3}) of Lemma \ref{li} and Proposition \ref{ls}(\ref{ls-2}), we have $u(\bar{b}_1)\neq 0$. Thus the image of $dc'_n$ in $\Gr_{m_n}L$ does not vanish. 
\end{proof}

\begin{exam}
  \begin{enumerate}
  \item
    We give an example for $n=2$, the simplest case where Theorem \ref{ta} can be applied.

    The algebraic group $G(2)$ is the Heisenberg group. The extension $M_2/K$ is defined by
  \begin{equation}
  \begin{pmatrix}
    1&x^p&{y_2}^p\\
    0&1&{y_1}^p\\
    0&0&1
  \end{pmatrix}
  =
  \begin{pmatrix}
    1&a&b_2\\
    0&1&b_1\\
    0&0&1
  \end{pmatrix}
  \begin{pmatrix}
    1&x&y_2\\
    0&1&y_1\\
    0&0&1
  \end{pmatrix}
  \end{equation}
  for some $a,b_1,b_2\in K$ satisfying the conditions of Lemma \ref{li}. We have
  \begin{equation}
  \omega_1=db_1,\;\omega_2=db_2-adb_1.
  \end{equation}
  By Theorem \ref{ta}, we have
  \begin{equation}
  U_{M_2/K}=U_{M_1/K}\cup\{r_2\}
  \end{equation}
  where
  \begin{equation}\label{ey}
  r_2=\max\left(-v_K(db_2-adb_1),\frac{m_a}{p}+m_1,m_a+\frac{m_1}{p}\right).
  \end{equation}
  We can calculate $U_{M_1/K}$ by Proposition \ref{ls}(\ref{ls-1}), since $M_1/K$ is an abelian extension.

\item We give an example where the maximum of the right-hand side of (\ref{ez}) is achieved by the first term or the second term, depending on the parameters $\eta,\eta'$.

  Let $p>2$, $n=p-1$, $\eta,\eta'\in\mathbb{Z}_{\geq 0}$. For $1\leq j\leq p-1$, define $f_j(x)\in\mathbb{Z}\left[\frac{1}{j!},x^{-1}\right]$ as a polynomial of $x^{-1}$ satisfying
  \begin{equation}
  \frac{df_j}{dx}=-\binom{x^{-\eta p-1}}{j-1}x^{-\eta'p-2}.
  \end{equation}
  Let $M_n/K$ be the Galois extension defined by
  \begin{equation}
\begin{pmatrix}
  \mathbf{A}(x^p)&F(y)\\
  0&1
\end{pmatrix}
=
\begin{pmatrix}
  \mathbf{A}(a)&b\\
  0&1
\end{pmatrix}
\begin{pmatrix}
  \mathbf{A}(x)&y\\
  0&1
\end{pmatrix},
\end{equation}
with $a= t^{-\eta p-1},\;b_j=\epsilon f_j(t)\;(1\leq j \leq n)$, where $t$ is a uniformizer of $K$ and $\epsilon\in k-\mathbb{F}_p$. These $a,b$ satisfy the conditions of Lemma \ref{li}. Then by Proposition \ref{ls}(\ref{ls-1}), we have $U_{M_1/K}=\{\eta p+1,\eta'p+1\}$.

We have $db=-\epsilon\mathbf{A}(t^{-\eta p-1})\mathbf{v}(1)t^{-\eta'p-2}dt$. Thus $\omega=-\epsilon\mathbf{v}(1)t^{-\eta'p-2}dt$. 
Thus we have 
$-v_K(\omega_1)=\eta'p+1$ and $-v_K(\omega_j)=-\infty$ for $2\leq j\leq p-1$.
Hence, by Theorem \ref{ta}, we have
\begin{equation}
  r_j=\max\left(\frac{(j-1)(\eta p+1)+p(\eta'p+1)}{p},\frac{(j+p-2)(\eta p+1)+\eta'p+1}{p}\right)
\end{equation}
\begin{equation}
  =(j-1)\left(\eta+\frac{1}{p}\right)+\eta'+1+(p-1)\max(\eta,\eta').
\end{equation}
Thus, if we have $\eta\geq\eta'$ (resp. $\eta\leq\eta'$), then the maximum of the right-hand side of (\ref{ez}) is achieved by the first (resp. second) term.
  \end{enumerate}
\end{exam}

\section*{Declarations}

No funding was received for conducting this study. The author has no financial or proprietary interests in any material discussed in this article. We do not analyse or generate any datasets, because our work proceeds within a theoretical and mathematical approach.

\bibliography{ramgr_04_resub_manumath}

\end{document}